\documentclass{amsart}
\usepackage{amsfonts,amssymb,amscd,amsmath,enumerate,verbatim,calc}
\newtheorem{theorem}{Theorem}[section]

\theoremstyle{definition}
\theoremstyle{definitions}
\newtheorem{definition}[theorem]{Definition}

\newtheorem{remark}[theorem]{Remark}

\theoremstyle{notations}
\theoremstyle{note}

\theoremstyle{remarks}

\newcommand{\T}{\mathrm}

\newcommand{\F}{\mathbb{F}}

\newcommand{\fa}{\frak{a}}

\newcommand{\fb}{\frak{b}}

\newcommand{\fp}{\frak{p}}

\begin{document}
\author[Z. Sepasdar]
{Zahra Sepasdar}

\title[two dimensional cyclic codes]
{Generator matrix for two-dimensional cyclic codes of arbitrary length}
\thanks{E-mail addresses:
zahra.sepasdar@mail.um.ac.ir and zahra.sepasdar@gmail.com}
\maketitle

\begin{center}
{\it
 Department of Pure Mathematics, Ferdowsi University of Mashhad,\\
P.O.Box 1159-91775, Mashhad, Iran} \\
\end{center}
\vspace{0.4cm}
\begin{abstract}
Two-dimensional cyclic codes of length $n=\ell s$ over the finite field $\F$ are ideals of the polynomial ring $\F[x,y]/{<x^s -1,y^\ell-1>}.$
The aim of this paper, is to present a novel method to study the algebraic structure of two-dimensional cyclic codes of any length $n=s\ell$ over the finite field $\F$. By using this method, we find the generator polynomials
for ideals of $\F[x,y]/{<x^s -1,y^\ell-1>}$ 
 corresponding to two dimensional cyclic codes. These polynomials will be applied to obtain the generator matrix for two-dimensional cyclic codes.
\vspace{0.5cm}
\end{abstract}

\section{Introduction}
One of the important generalizations of the cyclic code is two-dimensional cyclic (TDC) code. 
\begin{definition}
Suppose that $C$ is a linear code over $\F$ of length $s\ell$ whose codewords are viewed as $s×\ell$ arrays. That is every codeword $c$ in $C$ has the following form
\begin{equation*} 
c= \left(
\begin{array}{ccc}
c_{0,0} \ \ \ \dots \ \ c_{0,\ell-1}\\
c_{1,0} \ \ \ \dots \ \ c_{1,\ell-1}\\
\vdots\\
c_{s-1,0} \ \ \dots \ \ c_{s-1,\ell-1}
\end{array} \right).
\end{equation*}
If $C$ is closed under row shift and column shift of codewords, then we call $C$ a $\T{TDC}$ code of size $s\ell$ over $\F$.
\end{definition}
It is well known that TDC codes of length $n=s\ell$ over the finite field $\F$ are ideals of the polynomial ring $\F[x,y]/{<x^s -1,y^\ell-1>}.$ 

The characterization for TDC codes, for the first time was presented by Ikai et al. in \cite{ik}. Since the method was pure, it didn't help decode these codes. After that, Imai introduced basic theories for binary TDC codes (\cite{im}). The structure of some two-dimensional cyclic codes corresponding to the ideals of $\F[x, y]/<x^s-1,y^{2^k}-1>$ is characterized by the present author in \cite{zahra}.

 The aim of this paper, is to find the generator matrix for TDC codes of arbitrary length $n=s\ell$ over the finite field $\F$. To achieve this aim, we present a new method to characterize ideals of the ring $\F[x, y]/<x^s-1,y^\ell-1>$ corresponding to TDC codes, and find generator polynomials for these ideals. Finally, we use these polynomials to obtain the generator matrix for corresponding TDC codes.

\begin{remark}
For simplicity of notation, we write $g(x)$ instead of $g(x)+<\fa>$ for elements of $\F[x]/<\fa>$. Similarly, we write $g(x,y)$ instead of $g(x,y)+<\fa,\fb>$ for elements of $\F[x,y]/<\fa,\fb>$.
\end{remark}

\section{Generator polynomials}
Set $R:=\F[x,y]/<x^s-1, y^{\ell}-1>$ and $S:=\F[x]/<x^s-1>$. Suppose that $I$ is an ideal of $R$. In this section, we construct ideals $I_i$  of $S$ ($i=0,\dots, \ell-1$) and prove that the monic generator polynomials of $I_i$ provide a generating set for $I$.  Since
$$\F[x,y]/<x^s-1,y^{\ell}-1>\cong (\F[x]/<x^s-1>)[y]/<y^{\ell}-1>,$$
an arbitrary element $f(x,y)$ of $I$ can be written uniquely as $f(x,y)=\sum_{i=0}^{\ell-1}f_i(x)y^i$, where $f_i(x) \in S$ for $i=0,\dots, \ell-1$.
Put
\begin{align*}
I_0=\{g_0(x) \in S: \T{there \  exists \ } & g(x,y)\in I \ \T{such \  that}\ g(x,y)=\sum_{i=0}^{\ell-1}g_i(x)y^i\}.
\end{align*}
First, we prove that $I_0$ is an ideal of the ring $S$. Assume that $g_0(x)$ is an arbitrary element of $I_0$. According to the definition of $I_0$, there exists $g(x,y)\in I$ such that $g(x,y)=\sum_{i=0}^{\ell-1}g_i(x)y^i$. Now,  $xg_0(x)\in I_0$ since $I$ is an ideal of $R$ and $xg(x,y)=\sum_{i=0}^{\ell-1}xg_i(x)y^i$ is an element of $I$. Besides, $I_0$ is closed under addition and so $I_0$ is an ideal of $S$.
It is well-known that $S$ is a principal ideal ring. Therefore, there exists a unique monic polynomial $p_0^{0}(x)$ in $S$ such that $I_0=<p_0^{0}(x)>$ and $p_0^{0}(x)$ is a divisor of $x^s-1$. So there exists a polynomial $p'_0(x)$ in $\F[x]$ such that $x^s-1=p'_0(x)p_0^{0}(x)$. 
Now, consider the following equations
\begin{align*}
f(x,y)&=f_0(x)+f_1(x)y+\dots+ f_{\ell-1}(x)y^{\ell-1}\\
yf(x,y)&=f_0(x)y+f_1(x)y^2+\dots+ f_{\ell-1}(x)y^{\ell}\\&
=f_{\ell-1}(x)+f_0(x)y+f_1(x)y^2+\dots+f_{\ell-2}(x)y^{\ell-1}. \ \ \ \ \ (y^\ell=1 \ \T{in} \ R )
\end{align*}
Since $I$ is an ideal of $R$, $yf(x, y) \in I$. So according to the definition of $I_0$, $f_{\ell-1}(x) \in I_0$.
A similar method can be applied to prove that $f_i(x) \in I_0=<p_0^{0}(x)>$ for $i=1,\dots,\ell-2$.
 So 
\begin{align}\label{1}
f_i(x)=p_0^{0}(x) q_i(x)
\end{align}
for some $q_i (x) \in S$. 
Now, $p_0^{0}(x) \in I_0$ so according to the definition of $I_0$, there exists $\fp_0(x,y) \in I$ such that $$\fp_0 (x,y)=\sum_{i=0}^{\ell-1} p_i^{0}(x)y^i.$$ 
Again since $I$ is an ideal of $R$, $y^i\fp_0(x, y) \in I$ for $i=1,\dots,\ell-1$. So according to the definition of $I_0$, $p^0_i(x) \in I_0=<p_0^{0}(x)>$. Therefore, 
\begin{align*}
p^0_i(x)=p_0^{0}(x) t^0_i(x)
\end{align*}
for some $t^0_i (x) \in S$, and so 
$$\fp_0 (x,y)=p_0^{0}(x)+\sum_{i=1}^{\ell-1} p_0^{0}(x) t^0_i(x)y^i.$$
Set 
\begin{align*}
h_1(x,y):&=f(x,y)-\fp_0(x,y)q_0(x)=\sum_{i=0}^{\ell-1}f_i(x)y^i-q_0(x)\sum_{i=0}^{\ell-1} p_i^{0}(x)y^i\\&=
f_0(x)+\sum_{i=1}^{\ell-1}f_i(x)y^i-p_0^{0}(x) q_0(x)-q_0(x)\sum_{i=1}^{\ell-1} p_i^{0}(x)y^i\\&=\sum_{i=1}^{\ell-1}f_i(x)y^i-q_0(x)\sum_{i=1}^{\ell-1} p_i^{0}(x)y^i. \ \ \ \ \ \  \ \ \ \ \  \ \ \ \ \ \T{(by\ equation\ \ref{1})} 
\end{align*}
Since $f(x,y)$ and $\fp_0(x,y)$ are in $I$ and $I$ is an ideal of $R$, $h_1(x,y)$ is a polynomial in $I$. Also note that  $h_1(x,y)$ is in the form of $h_1(x,y)=\sum_{i=1}^{\ell-1} h_i^{1}(x)y^i$ for some $h_i^{1}(x) \in S$. Now, put
\begin{align*}
I_1=\{g_1(x) \in S:& \ \T{there \  exists} \ g(x,y)\in I \ \T{such \  that \ }  g(x,y)=\sum_{i=1}^{\ell-1}g_i(x)y^i\}.
\end{align*}
By the same method being applied for $I_0$, it can be proved that $I_1$ is an ideal of $S$. Thus, there exists a unique monic polynomial $p_1^1(x)$ in $S$ such that $I_1=<p_1^{1}(x)>$ and $p_1^{1}(x)$ is a divisor of $x^s-1$. Therefore, there exists a polynomial $p'_1(x)$ in $\F[x]$ such that $x^s-1=p'_1(x)p_1^{1}(x)$. Now,  $h_1(x,y) \in I$ so according to the definition of $I_1$, $h_1^{1}(x) \in I_1=<p_1^{1}(x)>$, and so 
\begin{align}\label{2}
h_1^{1}(x)=p_1^{1}(x) q_1(x)
\end{align}
 for some $ q_1(x) \in S$. 
And now, $p_1^1(x) \in I_1$ so according to the definition of $I_1$, there exists $\fp_1(x,y) \in I$ such that $$\fp_1(x,y)=\sum_{i=1}^{\ell-1} p_i^1(x)y^i.$$
Again since $I$ is an ideal of $R$, $y^i\fp_1(x, y) \in I$. So according to the definition of $I_0$, $p^1_i(x) \in I_0=<p_0^{0}(x)>$ for $i=1,\dots,\ell-1$. Therefore, 
\begin{align*}
p^1_i(x)=p_0^{0}(x) t^1_i(x)
\end{align*}
for some $t^1_i(x) \in S$, and so 
$$\fp_1 (x,y)=\sum_{i=1}^{\ell-1} p_0^{0}(x) t^1_i(x)y^i.$$
Set
\begin{align*}
h_2(x,y):&=h_1(x,y)-\fp_1(x,y)q_1(x)=\sum_{i=1}^{\ell-1} h_i^{1}(x)y^i-q_1(x)\sum_{i=1}^{\ell-1} p_i^1(x)y^i\\&=h_1^{1}(x)y+\sum_{i=2}^{\ell-1} h_i^{1}(x)y^i-p_1^{1}(x) q_1(x)y-q_1(x)\sum_{i=2}^{\ell-1} p_i^1(x)y^i\\&=\sum_{i=2}^{\ell-1} h_i^{1}(x)y^i-q_1(x)\sum_{i=2}^{\ell-1} p_i^1(x)y^i.\ \ \ \ \ \  \ \ \ \ \  \ \ \ \ \ \T{(by\ equation\ \ref{2})} 
\end{align*}
Since $h_1(x,y)$ and $\fp_1(x,y)$ are in $I$ and $I$ is an ideal of $R$, $h_2(x,y)$ is a polynomial in $I$ in the form of $h_2(x,y)=\sum_{i=2}^{\ell-1}h_i^{2}(x)y^i$ for some $h_i^{2}(x) \in S$.
Put
\begin{align*}
I_2=\{g_2(x) \in S: & \ \T{there \  exists \ }  g(x,y)\in I \ \T{such \  that \ } g(x,y)= \sum_{i=2}^{\ell-1} g_i(x)y^i\}.
\end{align*}
Again $I_2$ is an ideal of $S$, and so there exists a unique monic polynomial $p_2^{2}(x)$ in $S$ such that $I_2=<p_2^{2}(x)>$. Also  $p_2^{2}(x)$ is a divisor of $x^s-1$, and so there exists a polynomial $p'_2(x)$ in $\F[x]$ such that $x^s-1=p'_2(x)p_2^{2}(x)$. Now, $h_2(x,y) \in I$ so according to the definition of $I_2$, $h_2^{2}(x) \in I_2=<p_2^{2}(x)>$. So 
\begin{align}\label{3}
h_2^{2}(x)=p_2^{2}(x)q_2(x)
\end{align}
for some $q_2(x)\in S$. 
Besides,  $p_2^2(x) \in I_2$ so by definition of $I_2$, there exists $\fp_2(x,y) \in I$ such that $$\fp_2(x,y)=\sum_{i=2}^{\ell-1}p_i^2(x)y^i.$$
Again since $I$ is an ideal of $R$, $y^i\fp_2(x, y) \in I$. So according to the definition of $I_0$, $p^2_i(x) \in I_0=<p_0^{0}(x)>$. Therefore, 
\begin{align*}
p^2_i(x)=p_0^{0}(x) t^2_i(x)
\end{align*}
for some $t^2_i(x) \in S$, and so 
$$\fp_2 (x,y)=\sum_{i=2}^{\ell-1} p_0^{0}(x) t^2_i(x)y^i.$$
Set 
\begin{align*}
h_3(x,y):&=h_2(x,y)-\fp_2(x,y)q_2(x)=\sum_{i=2}^{\ell-1} h_i^{2}(x)y^i-q_2(x)\sum_{i=2}^{\ell-1} p_i^2(x)y^i\\&=h_2^{2}(x)y^2+\sum_{i=3}^{\ell-1} h_i^{2}(x)y^i-p_2^{2}(x) q_2(x)y^2-q_2(x)\sum_{i=3}^{\ell-1} p_i^2(x)y^i\\&=\sum_{i=3}^{\ell-1} h_i^{2}(x)y^i-q_2(x)\sum_{i=3}^{\ell-1} p_i^2(x)y^i.\ \ \ \ \ \  \ \ \ \ \  \ \ \ \ \ \T{(by\ equation \ \ref{3})} 
\end{align*}
 Therefore, $h_3(x,y)$ is a polynomial in $I$ in the form of $h_3(x,y)=\sum_{i=3}^{\ell-1}h_i^3(x)y^i$ for some $h_i^{3}(x) \in S$.
In the next step, we put
\begin{align*}
I_3=\{&g_3(x) \in S: \T{there \  exists \ } g(x,y)\in I \ \T{such \  that \ } g(x,y)=\sum_{i=3}^{\ell-1}g_i(x)y^i\}.
\end{align*}
The same procedure is applied to obtain polynomials
$$h_4(x,y),\dots,h_{\ell-2}(x,y),\fp_3(x,y),\dots,\fp_{\ell-2}(x,y)$$
in $I$ and polynomials
$q_3(x),\dots,q_{\ell-2}(x)$ in $S$ and 
 construct ideals $I_4,\dots,I_{\ell-2}$. Finally, we set $$h_{\ell-1}(x,y):=h_{\ell-2}(x,y)-\fp_{\ell-2}(x,y)q_{\ell-2}(x).$$ Thus, $h_{\ell-1}(x,y)$ is a polynomial in $I$ in the form of $h_{\ell-1}(x,y)=h_{\ell-1}^{\ell-1}(x)y^{\ell-1}$. Set
\begin{align*}
I_{\ell-1}=\{ g_{\ell-1}(x) \in S: & \ \T{there \  exists \ }  g(x,y)\in I \ \T{such \  that \ }  g(x,y)=g_{\ell-1}(x)y^{\ell-1}\}.
\end{align*}
Clearly $I_{\ell-1}$ is an ideal of $S$. Thus, there exists a unique monic polynomial $p_{\ell-1}^{\ell-1}(x)$ in $S$ such that $I_{\ell-1}=<p_{\ell-1}^{\ell-1}(x)>$ and $p_{\ell-1}^{\ell-1}(x)$ is a divisor of $x^s-1$ (there exists $p'_{\ell-1}(x)$ in $\F[x]$ such that $x^s-1=p'_{\ell-1}(x)p_{\ell-1}^{\ell-1}(x)$). Now, $h_{\ell-1}(x,y) \in I$ so according to the definition of $I_{\ell-1}$,  $h_{\ell-1}^{\ell-1}(x) \in I_{\ell-1}=<p_{\ell-1}^{\ell-1}(x)>$. So
\begin{align}\label{tip}
h_{\ell-1}^{\ell-1}(x)=q_{\ell-1}(x)p_{\ell-1}^{\ell-1}(x)
 \end{align}
for some $q_{\ell-1}(x) \in S$. 
And now, $p_{\ell-1}^{\ell-1}(x) \in I_{\ell-1}$ so according to the definition of $I_{\ell-1}$, there exists $\fp_{\ell-1}(x,y) \in I$ such that $\fp_{\ell-1}(x,y)=p_{\ell-1}^{\ell-1}(x)y^{\ell-1}.$ 
Therefore, by equation \ref{tip}
$$h_{\ell-1}(x,y)=h_{\ell-1}^{\ell-1}(x)y^{\ell-1}=q_{\ell-1}(x)p_{\ell-1}^{\ell-1}(x)y^{\ell-1}=q_{\ell-1}(x)\fp_{\ell-1}(x,y).$$
Again since $I$ is an ideal of $R$, $y\fp_{\ell-1}(x, y) \in I$. So according to the definition of $I_0$, $p^{\ell-1}_{\ell-1}(x) \in I_0=<p_0^{0}(x)>$. Thus,  
\begin{align*}
p^{\ell-1}_{\ell-1}(x)=p_0^{0}(x) t^{\ell-1}_{\ell-1}(x)
\end{align*}
for some $t^{\ell-1}_{\ell-1}(x) \in S$, and so 
$$\fp_{\ell-1} (x,y)=p_0^{0}(x) t^{\ell-1}_{\ell-1}(x)y^{\ell-1}.$$
Therefore, for an arbitrary element $f(x,y) \in I$ we show that
\begin{align*}
& h_1(x,y):=f(x,y)-\fp_0(x,y)q_0(x)\\&
h_2(x,y):=h_1(x,y)-\fp_1(x,y)q_1(x)\\&
h_3(x,y):=h_2(x,y)-\fp_2(x,y)q_2(x)\\&
\dots\\&
h_{\ell-1}(x,y):=h_{\ell-2}(x,y)-\fp_{\ell-2}(x,y)q_{\ell-2}(x)\\&
h_{\ell-1}(x,y)=q_{\ell-1}(x)\fp_{\ell-1}(x,y).
\end{align*}
So 
\begin{align*}
f(x,y)&=\fp_0(x,y)q_0(x)+\fp_1(x,y)q_1(x)+\fp_2(x,y)q_2(x)\\&+
\dots+\fp_{\ell-2}(x,y)q_{\ell-2}(x)+\fp_{\ell-1}(x,y)q_{\ell-1}(x).
\end{align*}
Since $\fp_{i}(x,y)\in I$ for $i=0,\dots,\ell-1$ and $f(x,y)$ is an arbitrary element of $I$ and $I$ is an ideal of $R$, we conclude that 
$$I=<\fp_0(x,y),\dots,\fp_{\ell-1}(x,y)>,$$
where $\fp_j (x,y)=\sum_{i=j}^{\ell-1} p_0^{0}(x) t^j_i(x)y^i$. 
So $\{\fp_0(x,y), \fp_1(x,y),\dots,\fp_{\ell-1}(x,y)\}$ is a set of  generating polynomials for $I$. 

In the next theorem, we introduce the generator matrix for TDC codes.
\begin{theorem}
Suppose that $I$ is an ideal of $\F[x,y]/<x^s-1,y^{\ell}-1>$ and is generated by $\{\fp_0(x,y),\dots,\fp_{\ell-1}(x,y)\}$, which obtained from the above method. Then the set
\begin{align*}
\{&\fp_0(x,y),x\fp_0(x,y),\dots,x^{s-a_0-1}\fp_0(x,y), \\ & \fp_1(x,y),x\fp_1(x,y),\dots,x^{s-a_1-1}\fp_1(x,y), \\ & \ \ \ \ \ \ \ \ \ \ \ \ \ \ \ \ \ \ \ \  \ \ \ \ \ \vdots \\ & \fp_{\ell-1}(x,y),x\fp_{\ell-1}(x,y),\dots,x^{s-a_{\ell-1}-1}\fp_{\ell-1}(x,y)\}
\end{align*}
forms an $\F$-basis for $I$, where  $a_i=\T{deg}(p_{i}^i(x))$.
\end{theorem}
\begin{proof}
Assume that $l_0(x),\dots, l_{\ell-1}(x)$ are polynomials in $\F[x]$ such that $\T{deg}(l_i(x))< s-a_i$ and $l_0(x) \fp_0(x,y)+\dots+l_{\ell-1}(x)\fp_{\ell-1}(x,y)=0$. These imply the following equation in $S$ $$l_0(x)  p_0^0(x)=0.$$ Therefore, 
$l_0(x) p_0^0(x)=s(x) (x^s-1)$ for some $s(x) \in \F[x]$. Now,  the degree of $x$ in the right side of this equation is at least $s$ but since $\T{deg}(p_0^0(x))=a_0$ and $\T{deg}(l_0(x))< s-a_0$, the degree of $x$ in the left side of this equation is at most $s-1$. So we get
$l_0(x)=0$. Similar arguments yield $l_i(x)=0$ for $i=1,\dots,\ell-1$.
\end{proof}

\section{Conclusion}
In this paper, we present a novel method for studying the structure of TDC codes of length $n=s\ell$. This leads to studying the structure of ideals of the ring $\F[x,y]/<x^s-1,y^{\ell}-1>$. By using the novel method, we obtain generating sets of polynomials and generator matrix for TDC codes.

 \end{document}